\documentclass[11pt]{amsart}

\usepackage[margin=1.4in]{geometry}
\usepackage{amsthm,amsmath,amsfonts,amssymb}
\usepackage{graphicx}
\usepackage{dsfont}

\usepackage[usenames,dvipsnames,svgnames,table]{xcolor}

\usepackage[colorlinks=true,
pdfstartview=FitV,linkcolor=ForestGreen,citecolor=ForestGreen,
urlcolor=blue]{hyperref}

\usepackage{hyperref}

\newtheorem{thm}{Theorem}[section]
\newtheorem{prop}[thm]{Proposition}

\newtheorem{lemma}[thm]{Lemma}
\newtheorem{defn}[thm]{Definition}
\theoremstyle{remark}
\newtheorem{remark}[thm]{Remark}
\theoremstyle{remarks}
\newtheorem{remarks}[thm]{Remarks}
\numberwithin{equation}{section}

\newcommand{\R}{\mathbb R}
\newcommand{\T}{\mathbb T}
\newcommand{\Sp}{\mathbb S}

\newcommand{\eps}{\varepsilon}
\newcommand{\grad} {\nabla}

\newcommand{\dd} {\; \mathrm{d}}

\newcommand{\dss}{\displaystyle}

\DeclareMathOperator{\supp}{supp}

\DeclareMathOperator{\PV}{\mbox{p.v.}}

\title{Gaussian lower bounds for the Boltzmann equation without cut-off}

\author{Cyril Imbert} \address[C.~Imbert]{CNRS \& Department of
  Mathematics and Applications,
  \'Ecole Normale Sup\'erieure (Paris) \\
  45 rue d'Ulm, 75005 Paris, France} \email{Cyril.Imbert@ens.fr}

\author{Cl\'ement Mouhot} \address[C.~Mouhot]{University of Cambridge,
  DPMMS, Centre for Mathematical Sciences, Wilberforce road, Cambridge
  CB3 0WA, UK} \email{C.Mouhot@dpmms.cam.ac.uk}

\author{Luis Silvestre} \thanks{LS is supported in part by NSF grants
  DMS-1254332 and DMS-1362525. CM is partially supported by ERC grant
  MAFRAN} \address[L.~Silvestre]{Mathematics Department, University of
  Chicago, Chicago, Illinois 60637, USA}
\email{luis@math.uchicago.edu}

\date{\today}

\begin{document}

\begin{abstract}
  The study of positivity of solutions to the Boltzmann equation goes
  back to Carleman~\cite{MR1555365}, and the initial argument of
  Carleman was developed in
  \cite{PuWe97,MR2153518,MR3375544,MR3356579}, but the appearance of a
  lower bound with Gaussian decay had remained an open question for
  long-range interactions (the so-called non-cutoff collision
  kernels). We answer this question and establish such Gaussian lower
  bound for solutions to the Boltzmann equation without cutoff, in the
  case of hard and moderately soft potentials, with spatial periodic
  conditions, and under the sole assumption that hydrodynamic
  quantities (local mass, local energy and local entropy density)
  remain bounded. The paper is mostly self-contained, apart from the
  $L^\infty$ upper bound and weak Harnack inequality on the solution established
  respectively in~\cite{luis,imbert2018decay} and~\cite{is}.
\end{abstract}

\maketitle

\setcounter{tocdepth}{1}
\tableofcontents

\section{Introduction}

\subsection{The Boltzmann equation}

We consider the \textit{Boltzmann equation}
(\cite{Maxwell1867,Boltzmann-1872})
\begin{equation}\label{eq:boltzmann}
  \partial_t f + v \cdot \grad_x f = Q (f,f)
\end{equation}
on a given time interval $I = [0,T]$, $T \in (0,+\infty]$,
$x$ in the flat torus $\mathbb{T}^d$, and $v \in \R^d$.

The unknown $f = f(t,x,v) \ge 0$ represents the time-dependent
probability density of particles in the phase space, and $Q(f,f)$ is
the \textit{collision operator}, i.e. a quadratic integral operator
modelling the interaction between particles:
\[ 
  Q (f,f) = \int_{\R^d} \int_{\Sp} \left[ f(v'_*)f(v') - f(v_*)f(v)
  \right] B(|v-v_*|,\cos \theta) \dd v_* \dd \sigma
\] 
where the pre-collisional velocities $v_*'$ and $v'$ are given by
\[ 
  v' = \frac{v+v_*}2 + \frac{|v-v_*|}2 \sigma \quad \text{ and } \quad
  v'_* = \frac{v+v_*}2 - \frac{|v-v_*|}2 \sigma
\] 
and the \textit{deviation angle} $\theta$ is defined by (see
Figure~\ref{fig:collision})
\[
  \cos \theta := \frac{v-v_*}{|v-v_*|} \cdot \sigma \qquad \left(
    \text{ and } \quad \sin (\theta/2) := \frac{v'-v}{|v'-v|} \cdot
    \sigma = \omega \cdot \sigma \right) .
\]

\begin{figure}
	\includegraphics[height=6.5cm]{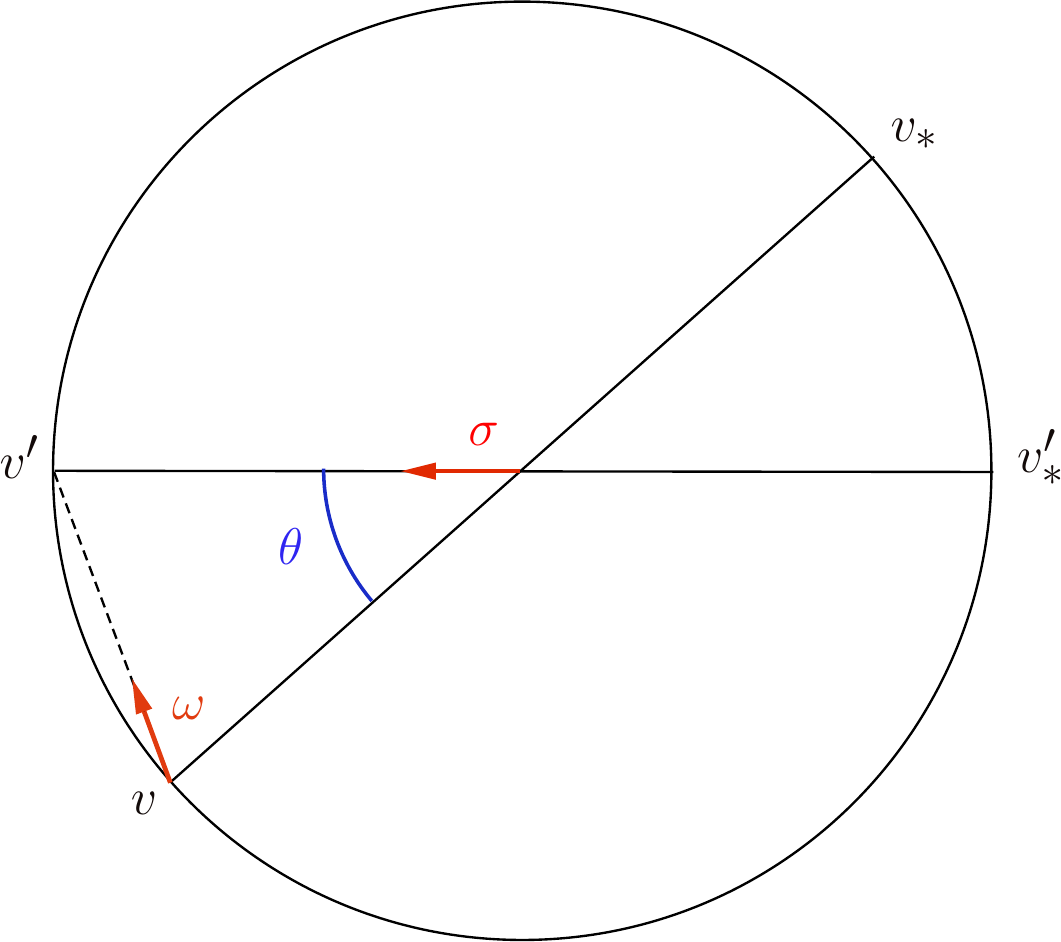}
	\caption{The geometry of the binary collision.}
	\label{fig:collision}
\end{figure}

It is known since Maxwell~\cite{Maxwell1867} that as soon as the
interaction between particles is long-range, the so-called
\textit{grazing collisions} are predominant, and this results in a
singularity of the collision kernel $B$ at small $\theta$. In
particular when particles interact microscopically via repulsive
inverse-power law potentials, the kernel $B$ has a non integrable
singularity around $\theta \sim 0$, commonly known as the
\emph{non-cutoff} case, and has the following general product form
\begin{equation}\label{assum:B}
	B(r,\cos \theta) = 
	r^\gamma b(\cos \theta) \quad \text{ with} \quad b (\cos \theta) 
	\approx_{\theta \sim 0} |\theta|^{-(d-1)-2s}
\end{equation}
with $\gamma > -d$ and $s \in (0,1)$. In dimension $d=3$ and for an
inverse-power law potential $\Phi(r) = r^{-\alpha}$ with
$\alpha \in (1,+\infty)$ then the exponents in \eqref{assum:B} are
\[
  \gamma = \frac{\alpha-4}{\alpha} \ \mbox{ and } \ s =
  \frac{1}{\alpha}.
\]
In dimension $d=3$, it is standard terminology to denote \emph{hard
  potentials} the case $\alpha >4$, \emph{Maxwellian molecules} the
case $\alpha =4$, \emph{moderately soft potentials} the case
$\alpha \in (2,4)$ and \emph{very soft potentials} the case
$\alpha \in (1,2)$. By analogy we denote in any dimension $d \ge 2$
\emph{hard potentials} the case $\gamma >0$, \emph{Maxwellian
  molecules} the case $\gamma=0$, \emph{moderately soft potentials}
the case $\gamma <0$ and $\gamma + 2s\in [0,2]$ and \emph{very soft
  potentials} the remaining case $\gamma <0$ and $\gamma + 2s \in
(-d,0)$.

\subsection{The program of conditional regularity}

The Boltzmann equation~\eqref{eq:boltzmann} is the main and oldest
equation of statistical mechanics. It describes the dynamics of a gas
at the \textit{mesoscopic level}, between the \textit{microscopic}
level of the many-particle (and thus very high dimension) dynamical
system following the trajectories of each particle, and the
\textit{macroscopic} level of fluid mechanics governed by Euler and
Navier-Stokes equations.

The dynamical system of Newton equations on each particle is out of
reach mathematically and contains way more information than could be
handled. Regarding the macroscopic level, the well-posedness and
regularity of solutions to the Euler and Navier-Stokes equations are
still poorly understood in dimension $3$ (with or without
incompressibility condition). The state of the art on the Cauchy
problem for the Boltzmann equation is similar to that of the $3$D
incompressible Navier-Stokes equations, which is not surprising given
that the Boltzmann equation ``contains'' the fluid mechanical equations
as formal scaling limits. Faced with this difficulty, Desvillettes and
Villani initiated in \cite{DeVi05} an \textit{a priori} approach,
where solutions with certain properties are assumed to exist and
studied. We follow this approach but refine it by assuming only
controls of natural local hydrodynamic quantities. This means that we
focus on the specifically kinetic aspect of the well-posedness issue.

Our result in this paper is conditional to the following bounds:
\begin{align}
  0 < m_0 \leq \int_{\R^d} f(t,x,v) \dd v
  &\leq M_0, \label{e:mass-density} \\
  \int_{\R^d} f(t,x,v) |v|^2 \dd v
  &\leq E_0, \label{e:energy-density} \\
  \int_{\R^d} f(t,x,v) \log f(t,x,v) \dd v
  &\leq H_0 \label{e:entropy-density}
\end{align}
for some constants $M_0>m_0>0$, $E_0>0$, $H_0>0$.

The first equation~\eqref{e:mass-density} implies that the mass
density is bounded above and below on the spatial domain and that
there is no vaccum. It would be desirable to relax its lower bound
part to only
$\int_{\mathbb T^d \times \R^d} f(t,x,v) \dd x \dd v \ge m_0>0$,
i.e. averaged in space. The equation~\eqref{e:energy-density} implies
that the energy density is bounded above on the spatial domain, and
the equation~\eqref{e:entropy-density} implies that the entropy
density is bounded above on the spatial domain (note that the energy
bound implies that it is also bounded below). These conditions are
satisfied for perturbative solutions close to the Maxwellian
equilibrium (see for instance the recent work~\cite{GMM} and
references therein in the hard spheres case and
\cite{GrSt11,AMUXY-maxw,AMUXY-hard,AMUXY-existence-soft,AMUXY-uniq-soft}
in the non-cutoff case) but it is an oustandingly difficult problem to
prove them in general.

Previous results of conditional regularity include, for the Boltzmann
equation and under the conditions above, the proof of $L^\infty$
bounds in~\cite{luis}, the proof of a weak Harnack inequality and
H\"older continuity in~\cite{is}, the proof of polynomially decaying
upper bounds in~\cite{imbert2018decay}, and the proof of Schauder
estimates to bootstrap higher regularity estimates
in~\cite{2018arXiv181211870I}. In the case of the closely related
Landau equation, which is a nonlinear diffusive approximation of the
Boltzmann equation, the $L^\infty$ bound was proved
in~\cite{luis-landau,gimv}, the Harnack inequality and H\"older
continuity were obtained in~\cite{MR2773175,gimv}, decay estimates
were obtained in~\cite{css}, and Schauder estimates were established
in~\cite{henderson2017c} (see also~\cite{2018arXiv180107891I}). The
interested reader is refereed to the short
review~\cite{2018arXiv180800194M} of the conditional regularity
program.

\subsection{The main result}

Let us define the notion of classical solutions we will use.
\begin{defn}[Classical solutions to the Boltzmann
  equation] \label{d:solutions} Given $T \in (0,+\infty]$, we say that
  a function
  $f : [0,T] \times \mathbb T^d \times \R^d \to [0,+\infty)$ is a
  \textbf{classical solution to the Boltzmann equation
    \eqref{eq:boltzmann}} if
  \begin{itemize}
  \item it is differentiable in $t$ and $x$ and twice
    differentiable in $v$ everywhere;
  \item the equation \eqref{eq:boltzmann} holds classically at every
    point in $[0,T] \times \mathbb T^d \times \R^d$.
  \end{itemize}
\end{defn}

The main result is then:

\begin{thm} \label{t:main} Assume that $\gamma+2s \in [0,2]$ (hard and
  moderately soft potentials). Let $f \geq 0$ be a solution
  to~\eqref{eq:boltzmann} according to Definition~\ref{d:solutions}
  that satisfies the hydrodynamic bounds
  \eqref{e:mass-density}-\eqref{e:energy-density}-\eqref{e:entropy-density}
  for all $t \in [0,T]$ and $x \in \mathbb{T}^d$. Then, there exists
  $a(t)>0$ and $b(t)>0$ depending on $t$, $s$, $\gamma$, $d$, $m_0$,
  $M_0$, $E_0$ and $H_0$ only so that
  \[ \forall \, t >0, \ x \in \T^d, \ v \in \R^d, \quad f(t,x,v) \geq
    a(t) e^{-b(t) |v|^2}.\]
\end{thm}
\begin{remarks}
  \begin{enumerate}
  \item The bound does not depend on the size domain of periodicity in
    $x$; this is due to the fact that the hydrodynamic bounds
    \eqref{e:mass-density}-\eqref{e:energy-density}-\eqref{e:entropy-density}
    are uniform in space. The periodicity assumption is made for
    technical conveniency and could most likely be removed.
  \item The requirement $\gamma + 2s \geq 0$ could be relaxed in our
    proof at the expense of assuming $f(t,x,v) \leq K_0$ for some
    constant $K_0>0$ (note that the functions $a(t)$ and $b(t)$ would
    then depend on $K_0$). However when $\gamma+2s <0$, this
    $L^\infty$ bound is only proved when assuming more than
    \eqref{e:mass-density}-\eqref{e:energy-density}-\eqref{e:entropy-density}
    on the solutions, namely some $L^\infty_{t,x}L^p_v$ bounds with
    $p>0$, see \cite{luis}.
  \end{enumerate}
\end{remarks}

\subsection{Previous results of lower bound and comparison}

The emergence and persistence of lower bounds for the Boltzmann
equation is one of the most classical problems in the analysis of
kinetic equations. It is a natural question: it advances the
understanding of how the gas fills up the phase space, and how it
relaxes towards the Maxwellian state; and such lower bounds are
related to coercivity properties of the collision operator.

The study of lower bounds was initiated in the case of short-range
interactions (namely hard spheres) and for spatially homogeneous
solutions: Carleman \cite{MR1555365} proved the generation of lower
bounds of the type $f(t,v) \ge C_1 e^{-C_2|v|^{2+\epsilon}}$ for
constants $C_1,C_2 >0$ and $t \ge t_0>0$, with $\epsilon >0$ as
$t_0 >0$ as small as wanted. He considered classical solutions with
polynomial pointwise estimates of decay (such estimates are also
proved in his paper) and assumed that the initial data has already a
minoration over a ball in velocity. This was later significantly
improved in \cite{PuWe97}: the authors proved in this paper that
spatially homogeneous solutions with finite mass, energy and entropy
are bounded from below by a Maxwellian $C_1 e^{-C_2|v|^2}$, where the
constants depend on the mass, energy and entropy; they obtained the
optimal Maxwellian decay by refining the calculations of Carleman but
also got rid of the minoration assumptions on the initial data through
a clever use of the iterated gain part of the collision kernel. Note
that, in this spatially homogeneous setting and for hard spheres, the
assumption of finite entropy could probably be relaxed in the latter
statement by using the non-concentration estimate on the iterated gain
part of the collision operator proved later in
\cite{MR1697495}. Finally, still for hard spheres, the optimal
Maxwelllian lower bound was extended to spatially inhomogeneous
solutions in the torus satisfying the hydrodynamic bounds
\eqref{e:mass-density}-\eqref{e:energy-density}-\eqref{e:entropy-density}
in \cite{MR2153518}, and to domains with boundaries in
\cite{MR3375544,MR3356579}. The paper \cite{MR2153518} also proved the
first lower bounds in the non-cutoff case, however they were poorer
than Maxwellian ($C_1 e^{-C_2|v|^\beta}$ with $\beta >2$ not
necessarily close to $2$) and required considerably stronger a priori
assumptions on the solutions than
\eqref{e:mass-density}-\eqref{e:energy-density}-\eqref{e:entropy-density}. The
latter point is due to the fact that the proof of the lower bound in
\cite{MR2153518} in the non-cutoff case is based on a decomposition of
the collision between grazing and non-grazing collisions and treating
the former as mere error terms. Thus, Theorem \ref{t:main} is a
significant improvement over the result of \cite{MR2153518} in the
non-cutoff case. Our proof here uses coercivity properties at grazing
collisions, as pioneered by \cite{luis}, instead of treating them as
error terms.

\subsection{Method of proof}

The well established pattern for proving lower bounds goes back to
Carleman \cite{MR1555365} and follows the collision process: (1)
establish a minoration on a ball on the time interval considered, (2)
spread iteratively this lower bound through the collision process,
i.e. using coercivity properties of the collision operator. The step
(1), i.e. the minoration on a ball $v \in B_1$, is deduced here from
the weak Harnack inequality as in \cite{is}, see
Proposition~\ref{p:localLB} below (and \cite[Theorem~1.3]{is}). The
spreading argument of step (2) is then performed in
Lemma~\ref{l:unit-spread}. The geometric construction in Lemma
\ref{l:unit-spread} resembles the iterative spreading of lower bounds
in the cut-off case, as in \cite{MR2153518}. The key difference is the
way we handle the singularity in the integral kernel. In
\cite{MR2153518}, a priori assumptions of smoothness of the solution
are used to remove a neighbourhood around $\theta=0$ in the collision
integral and treat it as an error term. Here instead we use coercivity
and sign properties of this singular part of the collision operators,
as developed and used in
\cite{luis,2018arXiv181211870I,imbert2018decay}; because of the
fractional derivative involved we use also a barrier method to justify
the argument; it is inspired from \cite{silvestre2014regularity}, and
was recently applied to the Boltzmann equation in
\cite{luis,imbert2018decay}.

\subsection{A note on weak solutions}

No well-posedness results are known for the Boltzmann equation without
perturbative conditions or special symmetry. This is true both for
strong and weak notions of solutions. As far as the \emph{existence}
is concerned, the unconditional existence of solutions is only known
for \emph{renormalized solutions with defect measure},
see~\cite{av2002}. Current results on \emph{uniqueness} of solutions
require significant regularity assumptions (see for example
\cite{amuxy-uniqueness}). It is thus not surprising that it is rather
inconvenient to prove estimates for the inhomogeneous Boltzmann
equation in any context other than that of classical solutions. Our
main result in Theorem \ref{t:main} is presented as an a priori
estimate on classical solutions. The estimate does not depend
quantitatively on the smoothness of the solution $f$. Some
computations in the proof however require a qualitative smoothness
assumption of $f$ so the quantities involved make sense.

Let us be more precise, and discuss the two parts of the proof of
Theorem~\ref{t:main} described in the previous section. Part~(1) is
established thanks to the weak Harnack inequality from \cite{is} (see
Proposition \ref{p:localLB}). The qualitative conditions necessary for
this step, as stated in \cite{is}, is that
$f \in L^2([0,T] \times \mathbb T^d, L^\infty_{loc} \cap
H^s_{loc}(\R^d))$ solves the equation \eqref{eq:boltzmann} in the
sense of distributions. Part~(2) consists in expanding the lower bound
from $v \in B_1$ to larger values of $|v|$ and is based on comparison
principles with certain barrier functions. The notion of solution that
is compatible with these methods is that of \emph{viscosity
  super-solutions}. In this context, it would be defined in the
following way. Denote $\underline f$ the lower semicontinuous envelope
of $f$. We say a function
$f: C([0,T] \times \mathbb T^d, L^1_2(\R^d))$ is a viscosity
super-solution of \eqref{eq:boltzmann} if whenever there is a $C^2$
function $\varphi$ for which $\underline f - \varphi$ attains a local
minimum at some point
$(t_0,x_0,v_0) \in (0,T] \times \mathbb T^d \times \R^d$, then the
following inequality holds 
\begin{align*} 
  (\varphi_t + v \cdot \nabla_x \varphi)(t_0,x_0,v_0)
  &\geq \int_{B_\eps(v_0)} (\varphi(t_0,x_0,v') - \varphi(t_0,x_0,v_0)) K_f(t_0,x_0,v_0,v') \dd v'\\
  &+ \int_{\R^d \setminus B_\eps(v_0)} (f(t_0,x_0,v') - f(t_0,x_0,v_0)) K_f(t_0,x_0,v_0,v') \dd v'\\
  & + Q_{ns}(f,\varphi)(t_0,x_0,v_0).
\end{align*}
Here $K_f$ is the Boltzmann kernel written in \eqref{eq:K}, $Q_{ns}$
is the non-singular term written below in \eqref{e:Qns}, $\eps>0$ is
any small number so that the minimum of $\underline f - \varphi$ in
$B_\eps(v_0)$ is attained at $v_0$. Following the methods
in~\cite{barles2008second} or \cite{silvestre2014regularity} for
instance, one should have no problem reproducing our proofs in this
paper in the context of such viscosity solutions. 

It is unclear how the notion of viscosity solution compares with the
notion of renormalized solutions. We are not aware of any work on
viscosity solutions in the context of the Boltzmann equation. If one
tries to adapt the proofs in this paper to the renormalized solutions
with defect measure of~\cite{MR1857879}, it seems that one would face
serious technical difficulties. In order to keep this paper cleaner,
and to make it readable for the largest possible audience, we believe
that it is most convenient to restrict our analysis to classical
solutions.

\subsection{Plan of the paper}

Section~\ref{sec:preliminaries} introduces the decomposition of the
collision operator adapted to the non-cutoff setting and recalls key
estimates on it. Section~\ref{sec:lower-bounds} proves the main
statement: we first recall the result of \cite{is} providing a
minoration on a ball, then introduce our new argument for the
spreading step, and finally complete the proof that follows readily
from the two latter estimates.

\subsection{Notation}

We denote $a \lesssim b$ (respectively $a \gtrsim b$) for $a \le C b$
(respectively $a \ge C b$) when the constant $C>0$ is independent from
the parameters of the calculation; when it depends on such parameters
it is indicated as an index, such as $a \lesssim_{M_0} b$. We denote
$B_R(v_0)$ the ball of $\R^d$ centered at $v_0$ and with radius $R$,
and we omit writing the center when it is $0$, as in $B_R=B_R(0)$. 

\section{Preliminaries}
\label{sec:preliminaries}

\subsection{Decomposition of the collision operator}

It is standard since the discovery of the so-called ``cancellation
lemma'' \cite{advw} to decompose the Bolzmann collision operator
$Q(f,f)$ into singular and non-singular parts as follows: 
\begin{align*}
  & Q (f_1,f_2)(v) \\
  & = \int_{\R^d \times \Sp} \Big[ f_1(v'_*) f_2(v') - f_1(v_*) f_2(v) \Big] B \dd v_* \dd \sigma \\
  & = \int f_1(v'_*) \big[ f_2(v')-f(v) \big] B \dd v_*
    \dd \sigma + f_2(v) \int \big[ f_1(v'_*) - f_1(v_*)
               \big] B \dd v_* \dd \sigma \\
  & =: Q_s(f_1,f_2) + Q_{ns}(f_1,f_2)
\end{align*}
where ``s'' stands for ``singular'' and ``ns'' stands for
``non-singular''. The part $Q_{ns}$ is indeed non-singular: Given
$v \in \R^d$, the change of variables
$(v_*,\sigma) \mapsto (v'_*,\sigma)$ has Jacobian
${\rm d} v'_* {\rm d}\sigma = 2^{d-1} (\cos \theta/2)^2 \, {\rm d} v_*
{\rm d}\sigma$, which yields (same calculation as
\cite[Lemma~1]{advw})
\begin{equation} \label{e:Qns}
  Q_{ns}(f_1,f_2)(v)  = f_2(v) \int_{\R^d} \int_{\Sp} \big[ f_1(v'_*) - f_1(v_*)
            \big] B \dd v_* \dd \sigma 
               =: f_2(v) (f_1 * S)(v)
\end{equation}
with
\begin{align*}
  S(u) & := \left|\mathbb S^{d-2}\right| \int_0 ^{\frac{\pi}{2}} (\sin
         \theta)^{d-2} \Big[ (\cos \theta/2)^{-d}
         B\left( \frac{|u|}{\cos
         \theta/2}, \cos \theta \right) - B(|u|,\cos \theta) \Big]
         \dd \theta\\
       & = \left|\mathbb S^{d-2}\right| |u|^\gamma
         \int_0 ^{\frac{\pi}{2}} (\sin
         \theta)^{d-2} \Big[ (\cos \theta/2)^{-d-\gamma}
         - 1 \Big]  b(\cos \theta)
         \dd \theta \\
       & =: C_S |u|^\gamma
\end{align*}
where we have used the precise form~\eqref{assum:B} of the collision
kernel in the second line. The constant $C_S >0$ is finite, positive
and only depends on $b$, $d$, and $\gamma$. The first term $Q_s$ is an
elliptic non-local integral operator of order $2s$ (see
\cite{advw,luis} and the many other subsequent works revealing this
fact), and the second term $Q_{ns}(f,f)$ is a lower order term that
happens to be nonnegative.

Observe that $Q_{ns} \geq 0$, and thus we can remove this lower order
term and the function $f$ is a supersolution of the following
equation,
\begin{equation} \label{e:Q1only}
 f_t + v \cdot \nabla_x f \geq Q_s(f,f),
\end{equation}
where
\[ Q_s(f,f)(v) = \int_{\R^d \times \mathbb S^{d-1}} \big[f_2(v') -
  f_2(v)\big] f_1(v'_*) b(\cos \theta) \dd v_* \dd \sigma.\] We change
variables (Carleman representation~\cite{MR1555365}) according to
$(v_*,\sigma) \mapsto (v',v'_*)$, with Jacobian
${\rm d} v_* {\rm d}\sigma = 2^{d-1} |v-v'|^{-1} |v-v_*|^{-(d-2)} \,
{\rm d} v' {\rm d} v'_*$ (see for instance \cite[Lemma~A.1]{luis}) and
we deduce
  \begin{equation}\label{eq:Carl-rep}
  Q_s(f_1,f_2)(v) = \PV \int_{\R^d} K_{f_1}(v,v') \big[ f_2(v')-f_2(v) \big] \dd v',
  \end{equation}
  where
  \begin{align} \nonumber
  K_{f_1}(t,x,v,v') & := \frac{2^{d-1}}{|v'-v|} \int_{v'_* \in \; v +
  	(v'-v)^\bot} f_1(t,x,v'_*) |v-v_*|^{\gamma-(d-2)} b(\cos \theta) \dd
  v'_* \\  \label{eq:K}
  & := \frac{1}{|v'-v|^{d+2s}} \int_{v'_* \in \; v +
  	(v'-v)^\bot} f_1(t,x,v'_*) |v-v'_*|^{\gamma+2s+1} \tilde b(\cos \theta) \dd
    v'_* 
  \end{align}
  and where we have used the assumption \eqref{assum:B} to write
\[
  2^{d-1} b(\cos \theta) = |v-v'|^{-(d-1)-2s} |v-v_*|^{(d-2)-\gamma}
  |v-v'_*|^{\gamma+2s+1} \tilde b(\cos \theta)
\]
with $\tilde b$ smooth on $[0,\pi]$ and stricly positive on
$[0,\pi)$. The notation $\PV$ denotes the Cauchy principal value
around the point $v$. It is needed only when $s \in [1/2,1)$.

\subsection{Estimates on the collision operator}

We start with a simple estimate from above the kernel $K_f$ as in
\cite[Lemma~4.3]{luis}.
\begin{prop}[Upper bounds for the kernel] \label{p:Kupper-bound}
  For any $r>0$, the following inequality holds
  \[
    \forall \, t \in [0,\infty), \ x \in \T^d,  \ v \in \R^d,
    \quad \left\{
      \begin{array}{l} \dss 
        \int_{B_r(v)} |v-v'|^2 K_f(t,x,v,v') \dd v' \lesssim
        \Lambda(t,x,v) r^{2-2s}, \\[3mm] \dss
            \int_{\R^d \setminus B_r(v)} K_f(t,x,v,v') \dd v' \lesssim
        \Lambda(t,x,v) r^{-2s}
      \end{array}
    \right.
  \]
  where
  \[
    \Lambda(t,x,v) = \int_{\R^d} f(t,x,v'_*) |v-v'_*|^{\gamma+2s} \dd v'_*.
  \]
\end{prop}

\begin{remark}\label{rk:Lambda}
  Note that if $\gamma+2s \in [0,2]$ and
  (\ref{e:mass-density}-\ref{e:entropy-density}) hold, then
\[
  \Lambda(t,x,v) \lesssim_{M_0,E_0} (1+|v|)^{\gamma+2s}.
\]
\end{remark}
\begin{proof}[Proof of Proposition~\ref{p:Kupper-bound}]
  To prove the first inequality we write (omitting $t,x$)
  \begin{align*}
    & \int_{B_r(v)} |v-v'|^2 K_f(v,v') \dd v'\\
    & = \int_{B_r(v)} |v-v'|^{2-d-2s} \int_{v'_* \in \; v +
      (v'-v)^\bot} f(v'_*) |v-v'_*|^{\gamma+2s+1} \tilde b(\cos \theta) \dd
      v'_*  \dd v' \\
    & = \int_{\omega \in \Sp^{d-1}} \left( \int_{u=0} ^r u^{1-2s} \dd u
      \right)
      \int_{v'_* \in \; v +
      (v'-v)^\bot} f(v'_*) |v-v'_*|^{\gamma+2s+1} \tilde b(\cos \theta) \dd
      v'_*  \dd v' \\
    & \lesssim r^{2-2s}  \int_{\omega \in \Sp^{d-1}}  \int_{\tilde \omega \in
      \Sp^{d-1}} \delta_0(\omega \cdot \tilde \omega) \int_{\tilde u
      = 0} ^{+\infty} \tilde u ^{d-1+\gamma+2s} f(\tilde u \tilde
      \omega) \dd \tilde u \dd \tilde \omega \dd \omega \\
    & \lesssim r^{2-2s} \int_{v'_* \in \R^d} f(v'_*) |v-v'_*|^{\gamma+2s} \dd v'_*.
  \end{align*}
  The proof of the second inequality is similar:
  \begin{align*}
    & \int_{\R^d \setminus B_r(v)} K_f(v,v') \dd v'\\
    & = \int_{\R^d \setminus B_r(v)} |v-v'|^{-d-2s} \int_{v'_* \in \; v +
      (v'-v)^\bot} f(v'_*) |v-v'_*|^{\gamma+2s+1} \tilde b(\cos \theta) \dd
      v'_*  \dd v' \\
    & = \int_{\omega \in \Sp^{d-1}} \left( \int_{u=r} ^{+\infty} u^{-1-2s} \dd u
      \right)
      \int_{v'_* \in \; v +
      (v'-v)^\bot} f(v'_*) |v-v'_*|^{\gamma+2s+1} \tilde b(\cos \theta) \dd
      v'_*  \dd v' \\
    & \lesssim r^{-2s}  \int_{\omega \in \Sp^{d-1}}  \int_{\tilde \omega \in
      \Sp^{d-1}} \delta_0(\omega \cdot \tilde \omega) \int_{\tilde u
      = 0} ^{+\infty} \tilde u ^{d-1+\gamma+2s} f(\tilde u \tilde
      \omega) \dd \tilde u \dd \tilde \omega \dd \omega \\
    & \lesssim r^{-2s} \int_{v'_* \in \R^d} f(v'_*) |v-v'_*|^{\gamma+2s} \dd v'_*.
  \end{align*}
  This concludes the proof.
\end{proof}

The latter bounds are useful to estimate $Q_s(f,\varphi)$ for a $C^2$
barrier function $\varphi$.

\begin{lemma}[Upper bound for the linear Boltzmann
  operator] \label{l:Lphi} Let $\varphi$ be a bounded, $C^2$ function
  in $\R^d$. The following inequality holds
  \[
    |Q_s(f,\varphi)| = \left\vert \mbox{\emph{p.v.}} \int_{\R^d} \big[\varphi(v') -
      \varphi(v)\big] K_f(t,x,v,v') \dd v' \right\vert \leq \Lambda(t,x,v)
    \|\varphi\|_{L^\infty(\R^d)}^{1-s} [\varphi]_{\dot C^2(\R^d)}^s
  \]
  where $\Lambda(t,x,v)$ is the same quantity as in Proposition
  \ref{p:Kupper-bound} and
  \[
    [\varphi]_{\dot C^2(\R^d)} := \sup_{v' \neq v} \frac{|\varphi(v') -
      \varphi(v) - (v'-v) \cdot \nabla \varphi(v)|}{|v'-v|^2} \lesssim
    \|\nabla^2 \varphi\|_{L^\infty(\R^d)}.
  \]
\end{lemma}

\begin{proof}
  We decompose the domain of integration in $Q_s(f,\varphi)$ between
  $B_r(v)$ and $\R^d \setminus B_r(v)$, for an arbitrary radius $r>0$
  to be specified later. Due to the symmetry of the kernel
  $K_f(t,x,v,v+w)=K_f(t,x,v,v-w)$, we have that
  \[
    \PV \int_{B_r(v)} (v'-v) \cdot \nabla \varphi(v) K_f(t,x,v,v') \dd
    v' = 0.
  \]
Therefore
\begin{align*} 
  & \left\vert \PV \int_{B_r(v)} [\varphi(v') - \varphi(v)]
  K_f(t,x,v,v') \dd v' \right\vert \\
  &= \left\vert \int_{B_r(v)} [\varphi(v') - \varphi(v) - (v'-v) \cdot \nabla \varphi(v)] K_f(t,x,v,v') \dd v' \right\vert, \\
  &\leq [\varphi]_{\dot C^2(\R^d)} \int_{B_r(v)} |v'-v|^2 K_f(t,x,v,v') \dd v', \\
  &\lesssim [\varphi]_{\dot C^2(\R^d)} \Lambda(t,x,v) r^{2-2s},
\end{align*}
where we have used Proposition \ref{p:Kupper-bound} in the last line. 

Regarding the rest of the domain $\R^d \setminus B_r(v)$, we have
\begin{align*} 
  \left\vert \int_{\R^d \setminus B_r(v)} [\varphi(v') - \varphi(v)]
  K_f(t,x,v,v') \dd v' \right\vert
  &\leq 2 \|\varphi\|_{L^\infty} \int_{\R^d \setminus B_r(v)} K_f(t,x,v,v') \dd v', \\
  &\lesssim \|\varphi\|_{L^\infty} \Lambda(t,x,v) r^{-2s}.
\end{align*}

Adding both inequalities above, we get
\[
  Q_1(f,\varphi) \lesssim \Lambda(t,x,v) \left( [\varphi]_{C^2(v)}
    r^{2-2s} + \|\varphi\|_{L^\infty} r^{-2s} \right).
\]
We conclude the proof choosing
$r := \left( \|\varphi\|_{L^\infty(\R^d)} / [\varphi]_{\dot C^2(\R^d)}
\right)^{1/2}$.
\end{proof}

\subsection{Pointwise upper bound on the solution}

The following $L^\infty$ bound is one of the main results in
\cite{luis}. We state the slightly refined version
in~\cite[Theorem~4.1]{imbert2018decay}, which includes in particular
the limit case $\gamma+2s=0$.

\begin{prop}[Global upper bound --
  \cite{luis,imbert2018decay}] \label{p:Linfty} Assume
  $\gamma+2s \in [0,2]$. Let $f \geq 0$ be a solution to
  \eqref{eq:boltzmann} according to Definition~\ref{d:solutions} that
  satisfies the hydrodynamic bounds
  (\ref{e:mass-density}-\ref{e:energy-density}-\ref{e:entropy-density})
  for all $t \in [0,T]$ and $x \in \mathbb{T}^d$.

  Then, there exists a non-increasing function $b(t)>0$ on $(0,T]$
  depending on $m_0$, $M_0$, $E_0$ and $H_0$ only so that
  \[
    \forall \, t \in (0,T], \ x \in \T^d, \ v \in \R^d, \quad f(t,x,v)
    \leq b(t).
  \]
\end{prop}

\section{Lower bounds}
\label{sec:lower-bounds}

Recall that we prove the appearance of a lower bound on a ball thanks
to a weak Harnack inequality, then spread it iteratively using the
mixing properties of the geometry of collision and coercivity
estimates on the collision operator.

\subsection{Weak Harnack inequality and initial plateau}

In \cite{is}, two of the authors obtain a weak Harnack inequality for
the linear Boltzmann equation. This weak Harnack inequality implies a
local lower bound for the nonlinear Boltzmann equation. It is stated
in the following proposition. Note that Proposition \ref{p:Linfty}
gives us control of $\|f\|_{L^\infty}$ in terms of the other
parameters.
\begin{prop}[Local minoration -- {\cite[Theorem~1.3]{is}}]
  \label{p:localLB}
  Let $f \geq 0$ be a solution to
  \eqref{eq:boltzmann} according to Definition~\ref{d:solutions} that
  is $L^\infty$ and satisfies the hydrodynamic bounds
  (\ref{e:mass-density}-\ref{e:energy-density}-\ref{e:entropy-density})
  for all $t \in [0,T]$ and $x \in \mathbb{T}^d$.

  Then, for any $R>0$, there is a nondecreasing function
  $a: (0,\infty) \to (0,\infty)$ depending on $s$, $\gamma$, $d$,
  $m_0$, $M_0$, $E_0$ and $H_0$ and $R$ only, such that
  \[
    \forall \, v \in B_R(0), \quad f(t,x,v) \geq a(t).
  \]
\end{prop}
\begin{remark}\label{rk:positive}
  Note that this result implies, in particular, that any solutions as
  in the statement satisfies $f(t,x,v) > 0$ for every
  $(t,x,v) \in (0,T] \times \mathbb T^d \times \R^d$.
\end{remark}
\begin{remark}
  \label{rk:very-singular}
  Note that this result holds for $\gamma+2s<0$ conditionally to the
  $L^\infty$ bound. However this $L^\infty$ bound is proved only when
  $\gamma+2s \in [0,2]$ (see~\cite{luis,imbert2018decay}). 
\end{remark}

\subsection{Spreading lemma around zero}

%
\begin{lemma}[Spreading around zero]
  \label{l:unit-spread}
  Consider $T_0 \in (0,1)$. Let $f\geq 0$ be a supersolution of
  \eqref{e:Q1only} so that
  (\ref{e:mass-density}-\ref{e:energy-density}) hold, and such that
  $f \ge \ell \, {\bf 1}_{v \in B_R}$ on $[0,T_0]$ for some $\ell >0$
  and $R \geq 1$. 

  Then, there is a constant $c_s >0$ depending only on $d$, $s$,
  $M_0$, $E_0$ (but not on $m_0$ or $H_0$) such that for any
  $\xi \in (0,1-2^{-1/2})$ so that $\xi^{q} R^{d+\gamma} \ell < 1/2$ with
  $q=d+2(\gamma+2s+1)$, one has
  \[
    \forall \, t \in [0,T_0], \ x \in \T^d, \ v \in B_{\sqrt 2 (1-\xi) R}, \quad
    f(t,x,v) \ge c_s \xi^{q} R^{d+\gamma} \ell^2
    \min\left(t, R^{-\gamma} \xi^{2s}\right).
  \]
\end{lemma}

The proof of Lemma \ref{l:unit-spread} combines, at its core, the
spreading argument of the cut-off case that goes back to Carleman
\cite{MR1555365}, used here in the form developed in \cite{MR2153518},
coercivity estimates on the collision operator at grazing collisions
developed in \cite{luis,is,imbert2018decay}, and finally a barrier
argument similar to \cite[Theorem 5.1]{silvestre2014regularity}.

\begin{proof}[Proof of Lemma~\ref{l:unit-spread}]
  Given $\xi \in (0,1-2^{-1/2})$ as in the statement, consider a
  smooth function $\varphi_\xi$ valued in $[0,1]$ so that
  $\varphi_\xi =1$ in $B_{\sqrt 2(1-\xi)}$ and $\varphi_\xi =0$
  outside $B_{\sqrt 2 (1-\xi/2)}$ and
  $\|D^2 \varphi_\xi\|_{L^\infty (\R^d)} \lesssim \xi^{-2}$. Define
  $\varphi_{R,\xi} (v) := \varphi_\xi(v/R)$. Observe that
  $\|\varphi_{R,\xi}\|_{L^\infty(\R^d)} = 1$ and
  $\|D^2 \varphi_{R,\xi}\|_{L^\infty(\R^d)} \lesssim (R \xi)^{-2}$.

  Apply Lemma \ref{l:Lphi} and Remark~\ref{rk:Lambda} to get
  \begin{equation}\label{e:Q1phi}
    \left|Q_1(f,\varphi_{R,\xi})\right| \lesssim \Lambda(t,x,v) (R
    \xi)^{-2s} \lesssim_{M_0,E_0} R^{\gamma+2s} (R \xi)^{-2s}
    \lesssim_{M_0,E_0} R^\gamma \xi^{-2s}. 
  \end{equation}
  
  Define the following barrier function
  \[
    \tilde \ell(t) := \alpha \xi^{q} R^{d+\gamma}
    \ell^2 \left( \frac{1-e^{-C R^\gamma \xi^{-2s} t}}{C R^\gamma
        \xi^{-2s}} \right) 
  \]
  for some $\alpha \in (0,1)$ to be chosen small enough later
  and where $C \ge 1$ is the constant in~\eqref{e:Q1phi} depending on
  $M_0$, $E_0$. Our goal is to prove that
  \[
    \forall \, t \in [0,T_0], \ x \in \T^d, \ v \in \R^d, \quad
    f \ge \tilde \ell(t) \varphi_{R,\xi}.
  \]

  Thanks to Remark~\ref{rk:positive} we can assume that $f>0$
  everywhere on $[0,T_0] \times \T^d \times \R^d$: it is true for any
  $t \in (0,T_0]$ and we can shift the solution $f$ in time by
  $f(t+\eps,x,v)$ and prove the lower bound independently of $\eps>0$.

  Let us prove that $f(t,x,v) > \tilde \ell(t) \varphi_{R,\xi}(v)$ for
  all $t \in [0,T_0] \times \mathbb T^d \times \R^d$. This inequality
  holds at $t=0$ since $\tilde \ell(0) = 0$. If the inequality was not
  true, then there would be a first crossing point
  $(t_0,x_0,v_0) \in [0,T_0] \times \mathbb T^d \times \supp \varphi$
  (the crossing point cannot be outside the support of $\varphi$ since
  $f>0$) so that
  $f(t_0,x_0,v_0) = \tilde \ell(t_0) \varphi_{R,\xi}(v_0)$ and
  $f(t,x,v) \ge \tilde \ell(t) \varphi_{R,\xi}(v)$ for all
  $t \in [0,t_0]$, $x \in \T^d$, $v \in \R^d$.

  The smallness condition imposed on $\xi$ in the statement implies
  that $\tilde \ell(t) \le \ell/2$ for all $t \in [0,T_0]$, and thus
  $v_0 \notin B_R$ since $\varphi \equiv 1$ and $f \ge \ell$ in
  $B_R$. Moreover $v_0 \in B_{\sqrt 2 R (1-\xi/2)}$ since $\varphi=0$
  outside the latter ball. The contact point also satisfies the
  extremality and monotonocity (in time) conditions
  $\nabla_x f(t_0,x_0,v_0) = 0$ and
  $\tilde \ell'(t_0) \varphi_{R,\xi}(v_0) \ge \partial_t
  f(t_0,x_0,v_0)$.

  Using the fact that $f$ is a supersolution of \eqref{e:Q1only}, we
  thus get
  \begin{equation}
    \label{e:barrier-init}
    \tilde \ell'(t_0) \varphi_{R,\xi}(v_0) \ge Q_s(f,f)(t_0,x_0,v_0).
  \end{equation}

  We decompose $Q_s(f,f)$ as
  \begin{align*} 
    & Q_s(f,f)(t_0,x_0,v_0)
    = Q_s\left(f,f-\tilde \ell(t) \varphi_{R,\xi}\right)(t_0,x_0,v_0)
      + \tilde \ell(t) Q_s\left(f,\varphi_{R,\xi}\right)(t_0,x_0,v_0) , \\ 
    &= \int_{\R^d}
      \left[f(t_0,x_0,v')- \tilde \ell(t_0) \varphi_{R,\xi}(v')
      \right] K_f(t_0,x_0,v_0,v') \dd v' +
      \tilde \ell(t_0) Q_s\left(f,\varphi_{R,\xi}\right)(t_0,x_0,v_0).  
  \end{align*}
  We omit $t_0,x_0$ in $f$ and $K_f$ from now on to unclutter
  equations. The barrier satisfies
  \[
    \tilde \ell'(t) = \alpha \xi^q R^{d+\gamma} \ell^2 - C
    R^\gamma \xi^{-2s} \tilde \ell(t),
  \]
  and plugging the last equation into \eqref{e:barrier-init}, and
  using \eqref{e:Q1phi}, gives
  \[
    \alpha \xi^q R^{d+\gamma} \ell^2 - C \xi^{-2s} R^\gamma \tilde
    \ell(t_0) \ge \int_{\R^d} \left[f(v')- \tilde \ell(t_0)
      \varphi_{R,\xi}(v') \right] K_f(v_0,v') \dd v'- C \xi^{-2s}
    R^\gamma \tilde \ell(t_0).
  \]
  We cancel out the last term and obtain
  \begin{align*} 
    & \alpha \xi^q R^{d+\gamma}  \ell^2 \gtrsim\\
    & \int_{\R^d} \int_{v_*' \in \, v_0 + (v'-v_0)^\perp}
      \left[f(v') - \tilde
      \ell(t_0) \varphi_{R,\xi}(v') \right] f(v'_*)
      \frac{|v'-v'_*|^{\gamma+2s+1}}{|v_0-v'|^{d+2s}} \tilde b(\cos
      \theta) \dd v_*' \dd v'.
  \end{align*}

  Since the integrand is non-negative, we bound the integral from
  below by restricting the domain of integration to $v'\in B_R$ and
  $v'_* \in B_R$ (balls centered at zero):
 \begin{align*}
   & \alpha \xi^q R^{d+\gamma}  \ell^2 \gtrsim \\
   & \int_{v' \in B_R} \int_{v'_* \in \, v_0 + (v'-v_0)^\bot}
     {\bf 1}_{B_R}(v'_*) \left[f(v')- \tilde
     \ell(t_0) \varphi_{R,\xi}(v') \right] f(v'_*)
     \frac{|v'-v'_*|^{\gamma+2s+1}}{|v_0-v'|^{d+2s}} \tilde b(\cos \theta)
     \dd v_\ast' \dd v'.
 \end{align*}
 On this domain of integration, we have $f(v'_\ast) \geq \ell$, and
 the assumption $\xi^{q} R^{d+\gamma} \ell < 1/2$ implies that
 $f(v')- \tilde \ell(t_0) \varphi_{R,\xi}(v') \geq \ell - \tilde \ell
 \geq \ell/2$, thus
\begin{equation*}
  \alpha \xi^q R^{d+\gamma} \ell^2
  \gtrsim \ell^2
  R^{-d-2s} \int_{v' \in B_R} \int_{v'_* \in \, v_0 + (v'-v_0)^\bot} {\bf 1}_{B_R}(v'_*)
  |v'-v'_*|^{\gamma+2s+1} \tilde b (\cos \theta)  \dd v_\ast' \dd v'.
\end{equation*}
Observe that since $|v_0| \in [R,\sqrt 2 R (1-\xi/2)]$, the volume of
$v' \in B_R$ such that the distance between $0$ and the line $(vv')$
is more than $R(1-\xi/2)$, is $O(R^d \xi^{(d+1)/2})$. For $v'$ in this
region $\mathcal{C}_R$ (see shaded region in Figure~\ref{f:1}) the
$(d-1)$-dimensional volume of $v'_* \in B_R$ such that
$ (v'_* - v_0) \perp (v'-v_0)$ is $O(R^{d-1} \xi^{(d-1)/2})$. Finally
removing the $v' \in B_{\xi^2 R}(v_0)$ (which does not change the
volume estimates above) ensures that $|v'-v'_\ast| \ge \xi^2
R$. Therefore,
\begin{equation*}
  \alpha \xi^q R^{d+\gamma} \ell^2
  \gtrsim \ell^2 R^{-d-2s} \int_{v' \in \mathcal{C}_{R}}
  \int_{v' \in \mathcal{C}_{R}^\ast} |v'-v'_*|^{\gamma+2s+1}
  \tilde b (\cos \theta) \dd v_\ast' \dd v' 
  \gtrsim  \xi^q R^{d+\gamma} \ell^2
\end{equation*}
with $q=d+2(\gamma+2s+1)$, and where we have used the deviation angles
$\theta \sim \pi/2$ (non-grazing collisions) for which $\tilde b$ is
positive.

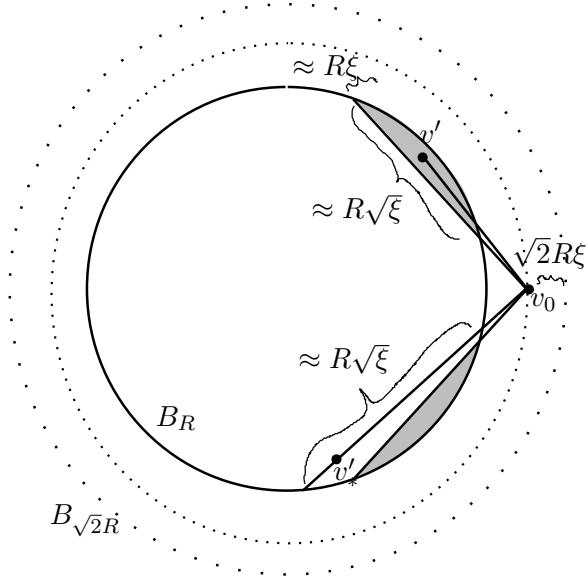
\begin{figure}[hbt] 
\setlength{\unitlength}{1in} 
\begin{picture}(2.926 ,3.000)
\put(0,0){\includegraphics[height=3.000in]{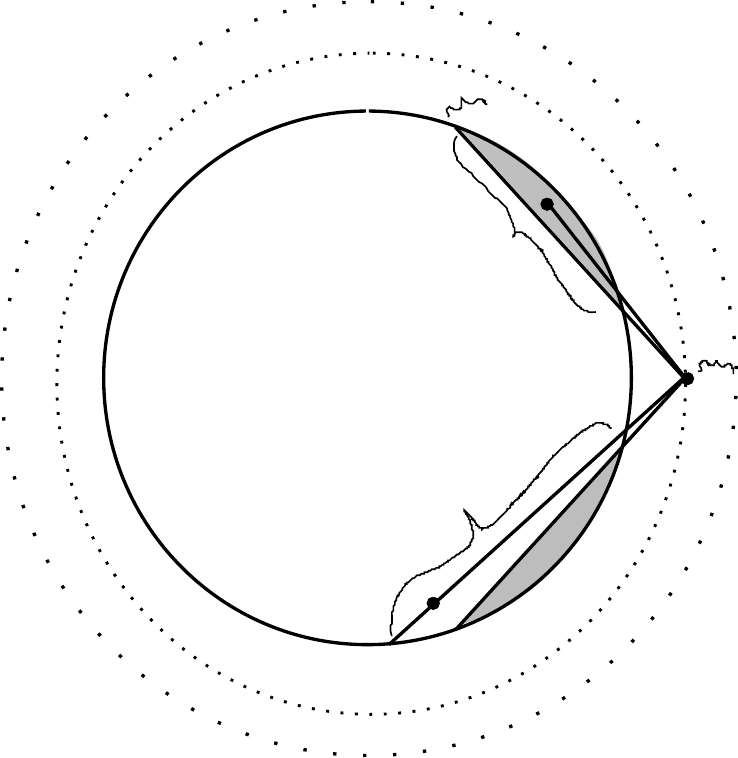}}
\put(2.151,2.269){$v'$}
\put(1.7,0.5){$v'_\ast$}
\put(2.73,1.411){$v_0$}
\put(0.771,0.785){$B_R$}
\put(0.217,0.272){$B_{\sqrt 2 R}$}
\put(1.487,2.621){$\approx R \xi$}
\put(2.636,1.637){$\sqrt 2 R \xi$}
\put(1.586,1.872){$\approx R \sqrt{\xi}$}
\put(1.518,1.088){$\approx R \sqrt{\xi}$}
\end{picture}
\caption{The binary collisions spreading the lower bound.}
\label{f:1}
\end{figure}

We have hence finally obtained the inequality
  \[
    \alpha \xi^q R^{d+\gamma} \ell^2 \ge \beta \xi^q
    R^{d+\gamma} \ell^2
  \]
  for some $\beta>0$ independent of the free parameter
  $\alpha \in (0,1)$, which is absurd.

  The resulting lower bound is 
  \begin{align*} 
    & \forall \, t \in [0,T_0], \ x \in \T^d, \ v \in B_{\sqrt 2(1-\xi) R}, \\
    & f(t,x,v)
     \ge \tilde \ell(t) \varphi_{R,\xi}(v) \\
    & \qquad \quad \ \, \, \geq \alpha \xi^q R^{d+\gamma} \ell^2
      \frac{1-e^{-C R^\gamma \xi^{-2s} t/2}}{CR^\gamma \xi^{-2s}}
    \ge c \xi^q R^{d+\gamma} \ell^2 \min\left(t, R^{-\gamma} \xi^{2s}\right)
\end{align*}
which concludes the proof.
\end{proof}

\begin{remark}
  The estimate in Lemma \ref{l:unit-spread} is most likely not optimal
  in terms of the power of $\xi$, as one could estimate better the
  factor $ |v'-v'_*|^{\gamma+2s+1} |v_0-v'|^{-d-2s} $. However we do not
  search for optimality here since the power in $\xi$ plays no role in
  the proof of Proposition \ref{prop:gaussian} below.
\end{remark}

\subsection{Proof of the Gaussian lower bound}
Theorem \ref{t:main} is a direct consequence of the following
proposition.
\begin{prop}[Gaussian lower bound]
  \label{prop:gaussian}
  Consider $T_0 \in (0,1)$. Assume that
  $f: [0,T_0] \times \mathbb T^d \times \R^d \to [0,+\infty)$ is a
  solution to \eqref{eq:boltzmann} according to
  Definition~\ref{d:solutions} that satisfies the hydrodynamic bounds
  (\ref{e:mass-density}-\ref{e:energy-density}-\ref{e:entropy-density})
  for all $t \in [0,T_0]$ and $x \in \mathbb{T}^d$.

  Then there are $a,b>0$ depending on $d$, $s$, $m_0$, $M_0$, $E_0$,
  $H_0$ and $T_0$ so that
  \[
    \forall \, x \in \T^d, v \in \R^d, \quad f(T_0,x,v) \ge a
    e^{-b|v|^2}.
  \]
\end{prop}
\begin{proof}
  Define the following sequences:
  \begin{equation*}
    \left\{
      \begin{array}{l}
        T_n := \left(1-\frac{1}{2^{n}} \right) T_0, \quad n \ge 1,\\[3mm]
        \xi_n = \frac{1}{2^{n+1}}, \quad n \ge 1,\\[3mm]
        R_{n+1} = \sqrt 2 (1-\xi_n) R_{n}, \quad n \ge 1, \quad R_0=1.
      \end{array}
    \right.
  \end{equation*}
  Observe that $2^{n/2} \lesssim R_n \le 2^{n/2}$ since
  $\Pi_{n=1}^{+\infty}(1-2^{-n}) < +\infty$.

  Proposition~\ref{p:localLB} implies that $f \ge \ell_0$ for
  $t \in [T_0/2,T_0]=[T_1,T_0]$, $x \in \T^d$, $v \in B_1=B_{R_0}$,
  and for some $\ell_0>0$, which initialises our induction.
  
  We then construct inductively a sequence of lower bounds $\ell_n>0$
  so that $f \geq \ell_n$ for $t \in [T_{n+1},T_0]$, $x \in \T^d$ and
  $v \in B_{R_n}$. We apply Lemma~\ref{l:unit-spread} repeatedly to
  obtain the successive values of $\ell_n$. Observe that
  $\xi_n^{q} R_n^{d+\gamma} \ell_n < (2^{-n})^{q - (d+\gamma)/2} <
  1/2$, so the smallness assumption on $\xi$ of Lemma
  \ref{l:unit-spread} holds through the iteration. The sequence of
  lower bounds $\ell_n$ satisfies the induction
\begin{align*} 
  \ell_{n+1}
  &= c_s \xi_n^{q} R_n^{d+\gamma} \ell_n^2
    \min \left(T_{n+1}-T_{n}, R_n^{-\gamma} \xi_n^{2s}\right), \\
  &= c_s \xi_n^{q} R_n^{d+\gamma} \ell_n^2 \min \left(2^{-n-1} T_0,
    R_n^{-\gamma} \xi_n^{2s}\right)
    \geq c 2^{-Cn} \ell_n^2 T_0
\end{align*}
for some constants $c,C>0$, which results in $\ell_n \ge u^{2^n}$ for
some $u \in (0,1)$. This implies the Gaussian decay.
\end{proof}
 
\begin{remark}
  Note that the proof of Lemma \ref{l:unit-spread} applies just as
  well in the cut-off case when $b$ is integrable, and covers actually
  all physical interactions. This is a manifestation of the fact that
  the collisions used to spread the lower bound are those with
  non-grazing angles $\theta \sim \pi/2$. In our notation, the
  short-range interactions correspond to $s<0$. The most important
  such short-range interaction is that of hard spheres in dimension
  $d=3$, corresponds to $\gamma=1$ and $s=-1$. Proposition
  \ref{p:localLB} is taken from \cite{is}, which applies exclusively
  to the non-cutoff case. In the proof of Theorem \ref{t:main}, we
  used Proposition \ref{p:localLB} to establish the lower bound in the
  initial ball $B_1$. This initial step would be different in the
  cut-off case. The estimates in the rest of the iteration carry
  through and the conclusion does not depend on $s$ being positive.
\end{remark}

\bibliographystyle{plain}
\bibliography{lower}
\end{document}